\documentclass[a4paper]{amsart}
\pdfoutput=1
\usepackage[section]{placeins}
\usepackage{url}
\usepackage{microtype}
\usepackage{booktabs, comment}
\usepackage[french, english]{babel}
\usepackage{amsmath, amssymb, amsthm, comment, enumerate, mathrsfs, multirow}
\usepackage{array}
\usepackage[colorlinks]{hyperref}
\hypersetup{citecolor=blue}
\usepackage{amscd}
\usepackage{xypic}
\usepackage{tikz-cd}
\usepackage{todonotes}

\theoremstyle{plain}
\newtheorem*{theorem}{Theorem}
\newtheorem{thm}{Theorem}[section]

\newtheorem{alg}[thm]{Algorithm}

\newtheorem{lem}[thm]{Lemma}

\newtheorem{rem}[thm]{Remark}

\newtheorem{quest}[thm]{Question}

\DeclareMathOperator{\PGL}{PGL}

\DeclareMathOperator{\N}{N}
\DeclareMathOperator{\Cl}{Cl}

\DeclareMathOperator{\Hom}{Hom}
\DeclareMathOperator{\GL}{GL}
\DeclareMathOperator{\SL}{SL}

\DeclareMathOperator{\Gal}{Gal}
\DeclareMathOperator{\Jac}{Jac}

\DeclareMathOperator{\End}{End}

\DeclareMathOperator{\GSp}{GSp}

\DeclareMathOperator{\Sp}{Sp}

\DeclareMathOperator{\Tr}{Tr}

\newcommand{\Q}{{\mathbf Q}}
\newcommand{\Z}{{\mathbf Z}}
\newcommand{\Zbar}{\overline{\Z}}
\newcommand{\C}{{\mathbf C}}
\newcommand{\R}{{\mathbf R}}
\newcommand{\F}{{\mathbf F}}

\newcommand{\CO}{\mathcal{O}}

\newcommand{\ga}{{\mathfrak{a}}}
\newcommand{\gb}{{\mathfrak{b}}}
\newcommand{\gc}{{\mathfrak{c}}}

\newcommand{\gm}{\mathfrak{m}}

\newcommand{\gP}{\mathfrak{P}}
\newcommand{\gp}{\mathfrak{p}}

\newcommand{\gN}{\mathfrak{N}}

\newcommand{\Qbar}{\overline{\Q}}

\newcommand{\Frob}{\mathrm{Frob}}

\title[Abelian surfaces with everywhere good reduction]{On the existence of abelian surfaces with everywhere good reduction}

\author{Lassina Demb\'el\'e}
\address{Department of Mathematics, Dartmouth College,
Hanover, NH 03755, USA}
\email{lassina.dembele@gmail.com}

\thanks{The author is supported by EPSRC Grants EP/J002658/1 and EP/L025302/1, and a Simons Collaboration Grant (550029).}
\begin{document}

\date{\today}

\begin{abstract}
Let $D \le 2000$ be a positive discriminant such that $F = \Q(\sqrt{D})$ has narrow class one, and $A/F$ an abelian
surface of $\GL_2$-type with everywhere good reduction. Assuming that $A$ is modular, we show that $A$ is either an 
$F$-surface or is a base change from $\Q$ of an abelian surface $B$ such that $\End_\Q(B) = \Z$, except for 
$D = 353, 421, 1321, 1597$ and $1997$. In the latter case, we show that
there are indeed abelian surfaces with everywhere good reduction over $F$ for $D = 353, 421$ and $1597$,
which are non-isogenous to their Galois conjugates. These are the first known such examples. 
\end{abstract}

\maketitle

\section{\bf Introduction}
The following is a well-known result due to Faltings~\cite[Satz 5]{fal83} (see also~\cite{fal84}):

\begin{thm}\label{thm:faltings} Let $F$ be a number field, $S$ a finite set of places of $F$, and $g \ge 1$ an integer. 
Then, the set of isomorphism classes of abelian varieties of dimension $g$ defined over $F$, with good reduction outside $S$, is finite.
\end{thm}
Theorem~\ref{thm:faltings} can be seen as an analogue of the Hermite-Minkowski theorem. 
The case when $S = \emptyset$ seems of particular interest since it relates to unramified motives. 
Fontaine~\cite{fon85} showed that there are {\it no} nonzero abelian varieties over $\Q$ with everywhere good 
reduction, thus proving Theorem~\ref{thm:faltings} for $F = \Q$, $S = \emptyset$ and all $g\ge 1$. 
Fontaine's result is very striking for two reasons at least. Indeed, not only is this one of the handful cases where one can explicitly determine the 
set of isomorphism classes of abelian varieties predicted by Theorem~\ref{thm:faltings}. But also, it shows that, for $F= \Q$ 
and $S = \emptyset$, this set is empty in every dimension $g\ge 1$. However, this non-existence result seems to be the 
exception rather than the norm. Indeed, Schoof~\cite{sch03} proved that, for $f \ge 1$ an integer not in $\{1, 3, 4, 5, 7, 8, 9, 11, 12, 15\}$, there exist 
non-zero abelian varieties with everywhere good reduction over the cyclotomic field $\Q(\zeta_f)$. Similarly, Moret-Bailly~\cite[Corollaire 5.9]{mob01} 
asserts that the stack $\mathscr{A}_g$ parametrising all principally 
polarised abelian schemes of dimension $g\ge 2$ over $\Z$ has a point over $\Zbar$. This means that, for every integer $g\ge 2$, there 
exist a number field $F$, and an abelian variety $A$ of dimension $g$ over $F$ with every good reduction. Therefore, in order to elucidate the
set of isomorphism classes predicted in Theorem~\ref{thm:faltings} for $S = \emptyset$, we might start with the following question:

\begin{quest}\rm \label{quest:ab-egr1} Given a number field $F$ and an integer $g\ge 1$, does there exist an abelian variety $A$ of
dimension $g$ defined over $F$, with everywhere good reduction?
\end{quest}

It is extremely difficult to give a purely arithmetic-geometric answer to Question~\ref{quest:ab-egr1} in general, even for $g = 1$, where there is 
a great deal of work over quadratic fields (see~\cite{cre92, elk14, kag97, kag01, kk97, pin82, set81, str83} for example). When $F$ is a 
real quadratic field, work of Freitas-Le Hung-Siksek~\cite{fls15} shows that all  elliptic curves defined over $F$ are modular. So, in this case, 
the set of isogeny  classes of elliptic curves with trivial conductor over $F$, corresponds to a subset of the set of Hilbert newforms of weight 
$2$, level $(1)$ and trivial central character on $F$, with integer Hecke eigenvalues. Similarly, the Eichler-Shimura conjecture predicts that the
latter set injects into the former, meaning that there is in fact a conjectural bijection between the two sets. So, for $F$ real quadratic field and $g=1$, 
one can provide an effective answer to Question~\ref{quest:ab-egr1} by first determining the set of Hilbert newforms of weight $2$, level $(1)$ 
and trivial central character on $F$, with integer Hecke eigenvalues. 

The Modularity conjecture for $\GL_2$-type abelian varieties and the Eichler-Shimura conjecture make similar predictions in every dimension 
$g\ge 1$. By making use of this, Kumar and the author found many examples of abelian surfaces with everywhere good reduction 
over real quadratic fields of narrow class number one and discriminant $\le 1000$ in~\cite{dk16}, thus providing an answer to 
Question~\ref{quest:ab-egr1} for most of those fields for $g = 2$. 
In this paper, we extend those results. More specifically, we proved the following result (Theorem~\ref{thm:egr}).

\begin{theorem} Let $F$ be a real quadratic field of narrow class number one and discriminant $D \le 2000$.
Let $A$ be a modular abelian surface of $\GL_2$-type defined over $F$, with everywhere good reduction.
Then, except for $D = 353, 421, 1321, 1597$ or $1997$, we have one of the following:
\begin{enumerate}[(i)]
\item A is an $F$-surface, i.e. there is an abelian fourfold $B$ of $\GL_2$-type defined over $\Q$ such that $B \times_\Q F$ is isogenous to $A \times {}^\sigma\!A$,
where $\Gal(F/\Q) = \langle \sigma \rangle$; or
\item There is an abelian surface $B$ defined over $\Q$ such that $\End_\Q(B) = \Z$ and $B \times_\Q F$ is isogenous to $A$.
\end{enumerate}
\end{theorem}

For the exceptional discriminants $D = 353, 421$ and $1597$, we showed that there are indeed abelian surfaces defined over $F$ with everywhere 
good reduction, which are  non-isogenous to their Galois conjugates; they are the first known such examples, and are dimension $2$ analogue 
of the elliptic curves of trivial conductor over $\Q(\sqrt{509})$ found by Pinch~\cite{pin82}. In~\cite{dk16}, the abelian surfaces were obtained by 
searching for rational points on Hilbert modular surfaces using explicit models in~\cite{ek14}. In this paper, we employed a height search in $2$-torsion 
fields, which in turn we used to refine the search methods in~\cite{dk16}.

Due to the nature of our approach, all our abelian surfaces are of $\GL_2$-type. It would be interesting to find a real quadratic field $F$, 
and an abelian surface $A$ defined over $F$ such that $A$ has trivial conductor with $\End_F(A) = \Z$. Such an abelian surface would be conjecturally 
attached to a Hilbert-Siegel eigenform of genus $2$, weight $2$ and level $(1)$, with integer Hecke eigenvalues. There are no methods for computing 
such forms yet, an added difficulty being that the weight $2$ is non-cohomological. However, recently, Chenevier~\cite{che18} proved a Hermite-Minkowski 
type theorem for automorphic forms over $\GL_n$. We believe that one could adapt his approach to $\GSp_4$ (and appropriate quadratic fields) to find 
the unramified Hilbert Siegel newforms needed to locate those abelian surfaces $A$ of trivial conductor, with $\End_F(A) = \Z$. 

The outline of the paper is as follows. In Section~\ref{sec:heights}, we start by revisiting the Doyle-Krumm algorithm for computing algebraic numbers 
of bounded height; we give a vastly improved version of the algorithm which could be of independent interest on its own. In Section~\ref{sec:2-tors-flds}, 
we review some background material on $2$-torsion of abelian surfaces, and in Section~\ref{sec:fontaine-bds}, we recall the Fontaine bounds for the 
root discriminants for the splitting fields of finite flat $p$-group schemes. In Section~\ref{sec:as-egr}, we describe all modular abelian surfaces of 
$\GL_2$-type over real quadratic fields with discriminant at most $2000$ and narrow class number one. Finally, in Sections~\ref{sec:sqrt353}, \ref{sec:sqrt421}, 
\ref{sec:sqrt1597} and \ref{sec:sqrt1997}, we discussed the missing surfaces.

{\bf Acknowledgements.} I would like to thank Armand Brumer, Fred Diamond, Steve Donnelly, Abhinav Kumar, Ren\'e Schoof and John Voight for helpful email 
exchanges and discussions. I would also like to thank Eric Driver, John Jones, J\"urgen Kluners and David P. Roberts for helping find some of the polynomials
displayed in this work. During the course of this project, I stayed at the following institutions: Dartmouth College, King's College London, and the Max-Planck Institute for 
Mathematics in Bonn; I would like to thank them for their generous hospitality.

\section{\bf Algebraic numbers of bounded height}\label{sec:heights}

In this section, we revisit the algorithm for computing algebraic numbers of bounded height described 
in the beautiful paper~\cite{dokr15} of Doyle-Krumm. We propose a refinement which makes the algorithm
significantly faster. 

Let $K$ be a number field, and $\CO_K$ the ring of integers of $K$. Let $\sigma_1,\,\ldots,\,\sigma_{r_1}$ be 
the real embeddings of $K$, and $\tau_1, \bar{\tau}_1,\,\ldots,\,\tau_{r_2},\bar{\tau}_{r_2}$ the  complex embeddings, 
so that $[K:\Q] = r_1 + 2r_2$. For each of these embeddings $\sigma$, the absolute $|\,\, |_\sigma$ is
given by $|x|_{\sigma} = |x|_{\C}$, where $|\,\,|_{\C}$ is the usual absolute value over $\C$, and that
$|\,\,|_{\tau_i} = |\,\,|_{\bar{\tau}_i}$, $i = 1,\,\ldots,\,r_2$. We let $M_K^\infty$ be the set of archimedian
absolute values. 

For a prime ideal $\gp$ of $\CO_K$, let $v_\gp:\,K \to \Z \cup \{\infty\}$ be the discrete valuation at $\gp$.
We recall that, for $x \in \CO_K$ nonzero, $v_\gp(x)$ is the largest integer $n\ge 0$ such that $\gp^n$ 
divides the ideal $(x)$. The absolute $|\,\,|_\gp$ on $K$ is defined by $|x|_\gp^{-v_\gp(x)/(e_\gp f_\gp)}$, where $e_\gp$
and $f_\gp$ are the ramification index and inertia degree respectively; it extends the $p$-adic absolute
value on $\Q$ where $p$ is the unique prime below $\gp$.  We let $M_K^0$ be the set of absolute 
values $|\,\,|_\gp$, and $M_K = M_K^\infty \cup M_K^0$.

For $v \in M_K$, let $K_v$ be the completion of $K$ at $v$, and $\Q_v$ the completion of $\Q$ 
at the restriction of $v$ to $\Q$. We let $n_v := [K_v: \Q_v]$ be the local degree at $v$. If $v$ is a real
place, then $K_v = \Q_v$ hence $n_v = 1$. If $v$ is a complex place, then $K_v = \C$
and $\Q_v = \R$, hence $n_v = 2$. If $v = v_\gp$ is a non-archimedean place, then $\Q_v = \Q_p$, 
where $p$ is the prime below $\gp$. In that case, $n_v = e_\gp f_\gp$. 

We define the {\it height} function $H_K:\,K \to \R_{>0}$ by 
$$H_K(x) = \prod_{v \in M_K} \max\{|x|_v^{n_v}, 1\}.$$
The function $H_K$ satisfies the following properties:
\begin{itemize}
\item For all $a, b \in K$, with $b \neq 0$,
$$H_K(a/b) =  \prod_{v \in M_K} \max\{|a|_v^{n_v}, |b|_v^{n_v}\};$$

\item For all $a, b \in K$, with $b \neq 0$,   
$$H_K(a/b) = \N_{K/\Q}(a,b)^{-1}\prod_{v \in M_K} \max\{|a|_v^{n_v}, |b|_v^{n_v}\},$$
where $\N_{K/\Q}(a,b)$ is the norm of the ideal generated by $a$ and $b$.

\item For all $a, b \in K$, $$H_K(ab) \le H_K(a)H_K(b).$$

\item For any $x \in K^\times$, $$H_K(x) = H_K(1/x).$$

\item For any $x \in K$, and $\zeta \in \mu_K$ a root of unity, $$H_K(\zeta x) = H_K(x).$$
\end{itemize}

For $x \in K$ non-zero, we define the {\it numerator ideal} and {\it denominator ideal} of $x$
\begin{align*}
\ga := \prod_{v_\gp(x) > 0} \gp^{v_\gp(x)},\,\,\text{and}\,\,\, \gb := \prod_{v_\gp(x) < 0} \gp^{-v_\gp(x)},
\end{align*} 
respectively. We note that $\gcd(\ga, \gb) = 1$, and that $\ga$ and $\gb$ belong to the same ideal class since 
$(x) = \ga \gb^{-1}$.

\begin{lem}\label{lem:nd-ideals} Let $x \in K^\times$, and let $\ga$ and $\gb$ the numerator and denominator ideals of $x$ respectively. 
If $H_K(x) \le B$, then we have $\N_{K/\Q}(\ga), \N_{K/\Q}(\gb) \le B$. 
\end{lem} 

\begin{proof} By definition, we have
\begin{align*}
H_K(x) &= \prod_{v} \max\{1, |x|_v^{n_v}\} = \prod_{v\mid \infty} \max\{1, |x|_v^{n_v}\} \prod_{v\nmid\infty} \max\{1, |x|_v^{n_v}\}\\
&=\prod_{v\mid \infty} \max\{1, |x|_v^{n_v}\} \prod_{\gp\nmid\infty \atop v_\gp(x)< 0} \max\{1, \N_{K/\Q}(\gp)^{-v_\gp(x)} \} \\
&\qquad\qquad{} \times \prod_{\gp\nmid\infty \atop v_\gp(x)>0} \max\{1, \N_{K/\Q}(\gp)^{-v_\gp(x)}\}\\
&= \N_{K/\Q}(\gb)\prod_{v\mid \infty} \max\{1, |x|_v^{n_v}\}.
\end{align*}
Therefore, $H_K(x) \le B$ implies that $\N_{K/\Q}(\gb) \le B$. Since $H_K(x) = H_K(1/x)$, we also get that $H_K(x) \le B$ implies that $\N_{K/\Q}(\ga) \le B$. 
\end{proof}

For a given positive real number $B$, the algorithm below computes all the elements $x \in K$ such that $H_K(x) \le B$.

\begin{alg}\label{alg:bd-height}\rm Given a number field $K$, and a positive real number $B$, this output all elements $x \in K$ such that $H_K(x) \le B$.
\begin{enumerate}
\item Compute the list $\mathscr{L}$ of all integral ideals $\gc$ such that $\N_{K/\Q}(\gc) \le B$;
\item Initiate the list $\mathscr{G} = \emptyset$. For each $\ga,\gb \in \mathscr{L}$ such that $\ga \gb^{-1}$ is principal and $\gcd(\ga, \gb) = 1$,
find a generator $\xi$ and append to $\mathscr{G}$;
\item Let $B_0 := \max\{H_K(\xi): \xi \in \mathscr{G}\}$, and compute the set $\mathscr{U}$ of units $u \in \CO_K^\times$ such that $H_K(u) \le B\cdot B_0$;
\item Initiate $\mathscr{S} = \{0\}$. For each $\xi \in \mathscr{G}$ such that $H_K(u\xi) \le B$, append $u\xi $ to $\mathscr{S}$.
\item Return $\mathscr{S}$. 
\end{enumerate}
\end{alg}

\begin{thm}\label{thm:bd-height} Given as input a number field $K$, and a positive integer $B$, Algorithm~\ref{alg:bd-height} outputs the set $\mathscr{S}$ of
all elements $x \in K$ such that $H_K(x) \le B$. 
\end{thm}

\begin{proof} Let $x \in K^\times$ such that $H_K(x) \le B$. Let $\ga$ and $\gb$ be the numerator and denominator ideals of $x$. Then by Lemma~\ref{lem:nd-ideals},
we have $\N_{K/\Q}(\ga), \N_{K/\Q}(\gb) \le B$. So $\ga$ and $\gb$ belong to $\mathscr{L}$. Let $\xi$ be the generator of $\ga\gb^{-1}$ contained in $\mathscr{G}$.
Then, $(x) = (\xi) = \ga\gb^{-1}$. Therefore, there is a unit $u \in \CO_K^\times$ such that $x = u \xi$. It remains to show that $H_K(x) \le B$ implies that 
$H_K(u) \le B\cdot B_0$. This follows from $u = x \xi^{-1}$, and the third and fourth properties of height functions. 
\end{proof}

Algorithm~\ref{alg:bd-height} is a slight variation on Algorithm 1 and its refinements described in~\cite{dokr15}.  In~\cite{dokr15}, one computes the set of units 
$\mathscr{U}$ by enumerating rational points in a polytope. However, that process becomes extremely slow as soon as the degree of the field exceeds $4$. 
One can substantially improve on this by using the following lemma. Let $r = r_1 + r_2 - 1$, and $\lambda: \CO_K^\times \to \R^{r+1}$ 
be defined by $$\lambda(u) := (\log |u|_{\sigma_1},\ldots,\log|u|_{\sigma_{r_1}}, \log|u|_{\tau_1}^2,\ldots,\log|u|_{\tau_{r_2}}^2).$$
Let $\|\cdot\|: \R^{r+1} \to \R_{\ge 0}$ denote the Euclidean norm. 

\begin{lem}\label{lem:units} Let $x \in \CO_K^\times$ be such that $H_K(u) \le B$. Then, we have
$$\|\lambda(u)\|^2 \le 2(\log B)^2.$$ 
\end{lem}

\begin{proof} We have

\begin{align*}
\log H_K(u) &= \sum_{v \mid \infty \atop |u|_v > 1 } \log |u|_v^{n_v} = -\sum_{v \mid \infty \atop |u|_v < 1}\log |u|_v^{n_v}.
\end{align*}
So, if $H_K(u) \le B$, then we have that
\begin{align*}
\|\lambda(u)\|^2 = \sum_{v \mid \infty}\left(\log|u|_v^{n_v}\right)^2 \le 2\left(\sum_{v \mid \infty\atop |u|_v > 1}\log |u|_v^{n_v}\right)^2 \le 2(\log B)^2. 
\end{align*}
\end{proof}

\begin{rem}\rm First, let $\gc_i$, $i = 1,\ldots, h$, be a complete set of representatives for the classes in $\Cl(K)$. If, $\gc$ is an integral ideal such that $\N_{K/\Q}(\gc) \le B$,
then there exists $\xi_i \in K^\times$, which is $\gc_i$-integral, such that $\gc = \xi_i \gc_i$ and $\N_{K/\Q}(\xi_i) \le \frac{B}{\N_{K/\Q}(\gc_i)}$. 
So, in practice, one does a class group precomputation for more efficiency. Then one combines Steps (1) and (2) by listing all the $\xi \in K^\times$ 
such that $\N_{K/\Q}(\xi) \le B$, and $\xi$ is $\gc_i$-integral for some $i = 1,\ldots, h$.

Second, the output set $\mathscr{S}$ of Algorithm~\ref{alg:bd-height} tends to be big as the degree of the field $K$ or the height bound $B$ 
grows. In practice, we found that it was more useful to have a variant of the algorithm which enumerates elements with a fixed denominator ideal $\gb$.
One can then vary the denominator if needed.

Finally, by Lemma~\ref{lem:units}, the problem of enumerating the unit set $\mathscr{U}$ in Step (3) becomes one of enumerating lattice points. This can be done very efficiently 
by using LLL algorithms, leading to substantial improvements in Step (3) of Algorithm~\ref{alg:bd-height}, which make the overall algorithm much faster. There is a 
great level of care and details in the algorithms described in~\cite{dokr15} and implemented in \verb|Sage|. The variations we proposed here will add significantly to
their efficiency as demonstrated by our own implementation.
\end{rem}

\section{\bf Galois representations attached to abelian surfaces}\label{sec:2-tors-flds}
In this section, we recall some useful results on the Galois representation on the $2$-torsion points of an abelian surface. 
We start with the following well-known lemma whose proof we couldn't find in the literature. 

\begin{lem}\label{lem:groups} Let $\F_2[\varepsilon] = \F_2[x]/(x^2)$. Then, there is a group isomorphism $\phi:\, \SL_2(\F_2[\varepsilon]) \simeq \Z/2\Z \times S_4$,
which identifies $\ker(\SL_2(\F_2[\varepsilon]) \to \SL_2(\F_2)) $ with the unique normal subgroup $N \simeq (\Z/2\Z)^3$ of $\Z/2\Z \times S_4$ of order $8$. 
\end{lem}

\begin{proof} As a subgroup of $S_6$, $\Z/2\Z\times S_4$ is generated by the two permutations $\sigma := (1, 2, 4, 5)$ and $\tau := (1, 3)(4, 6)$.
Similarly, the group $\SL_2(\F_2[\varepsilon])$ is generated by the matrices
\begin{align*}
A := \begin{pmatrix} \varepsilon + 1 & \varepsilon \\ 1 & 1\end{pmatrix},\,\,
B := \begin{pmatrix} 0 & 1\\ 1 & 0\end{pmatrix}. 
\end{align*}
The map $(\sigma \mapsto A,\, \tau\mapsto B)$ gives the desired isomorphism. 
\end{proof}

The following result is often attributed to Mumford, though it was already known in the 19th century.  

\begin{thm}
Let $C:\,y^2 = R(x)$ be a curve of genus $2$, and $A = \text{Jac}(C)$ its Jacobian. 
Let $\bar{\rho}_{A, 2}: \Gal(\Qbar/F) \to \Sp_{4}(\F_2)$ be the mod $2$ Galois representation attached to $A$,
and $L$ the fixed field of $\ker(\bar{\rho}_{A,2})$. Then, $L$ is the  splitting field of the polynomial $R(x)$.
\end{thm}


\begin{thm}\label{thm:2-torsion-grps} Let $A$ be an abelian surface defined over a field $k$ of characteristic zero, which has RM
by the maximal order $\CO$ of some quadratic field. Also let $\bar{\rho}_{A, 2}: \Gal(\Qbar/F) \to \GL_2(\CO\otimes \F_2)$ be the 
mod $2$ Galois representation attached to $A$, and $G=\mathrm{im}(\bar{\rho}_{A,2})$. Then,
we have the following:
\begin{enumerate}[(i)] 
\item If $2$ is inert in $\CO$ then $G\hookrightarrow A_5$.
\item If $2$ is split in $\CO$ then $G\hookrightarrow S_3\times S_3$.
\item If $2$ is ramified in $\CO$, then $G$ is a subgroup of $S_4 \times \Z/2\Z$, and
there is an exact sequence $0\to H\to G\to H'\to 0$ where
$H\hookrightarrow (\Z/2\Z)^3$ and $H'\hookrightarrow S_3$. In fact, we have $G \simeq H \rtimes H'$.
\end{enumerate}
\end{thm}

\begin{proof} The first and second statements follow from Wilson~\cite[Corollary 4.3.4]{wil98}.
To prove the third, we note that, since $2$ is ramified in $\CO$, $\CO \otimes \Z_2 \simeq \F_2[\varepsilon]$, with $\varepsilon^2 = 0$.
Then, we conclude by combining Lemma~\ref{lem:groups} and \cite[Corollary 4.3.4]{wil98}.
\end{proof}

\section{\bf Fontaine bound for finite group schemes}\label{sec:fontaine-bds}

The following theorem plays an important r\^ole in the proof of Theorem~\ref{thm:faltings} for $F = \Q$ and $S = \emptyset$ 
by Fontaine. It will be essential to us through out the paper.

\begin{thm}[Fontaine~\cite{fon85}]\label{thm:fontaine-bds} Let $p\ge 2$ be a prime, $F$ a number field and $A$ an abelian variety over $F$.
Assume that $A$ has everywhere good reduction. Let $K = F(A[p])$ be the field generated by the $p$-torsion points of $A$. Then, we have
$$\delta_K < \delta_F p^{1+\frac{1}{p-1}},$$
where $\delta_F$ and $\delta_K$ are the root discriminants of $F$ and $K$. 
\end{thm}

\begin{rem}\rm When $F$ is Galois, a much stronger statement than Theorem~\ref{thm:fontaine-bds} is true. In that case, let $L$ 
be the normal closure of $K = F(A[p])$. Then, we have 
$$\delta_L < \delta_F p^{1+\frac{1}{p-1}}.$$
This is proved in the same way as \cite[Lemme 3.4.2]{fon85}. 
\end{rem}

\section{\bf Abelian surfaces with everywhere good reduction}\label{sec:as-egr}

From now on, $F$ is a real quadratic field, with ring of integers $\CO_F$. Let $\gN$ be an integral ideal of $F$, and $f$ be a Hilbert 
newform of weight $2$ and level $\gN$, with Hecke eigenvalue field $K_f = \Q(a_{\gm}(f): \gm \subseteq \CO_F)$,
where $a_{\gm}(f)$ is the Hecke eigenvalue of the Hecke operator at $\gm$. We recall that for every $\tau \in \Hom(K_f, \Qbar)$, there is
a Hilbert newform $f^\tau$ determined by its Hecke eigenvalues by
$$a_{\gm}(f^\tau) := \tau(a_{\gm}(f)).$$
Similarly, there exists a Hilbert newform ${}^\sigma\!f$ of weight $2$ and level $\sigma(\gN)$ determined by its Hecke eigenvalues by
$$a_{\gm}({}^\sigma\!f) := a_{\sigma(\gm)}(f).$$
We recall that the $L$-series of $f$ is given by
$$L(f, s) := \sum_{\gm \subseteq \CO_F} \frac{a_{\gm}(f)}{\mathrm{N}(\gm)^{s}}.$$

We recall that an abelian surface $A$ is said to be of {\it $\GL_2$-type} if there exists a quadratic field $K$ such that $\End_F(A)\otimes \Q \simeq K$. 
In that case, $A$ is said to be {\it modular} if there exists a Hilbert newform of weight $2$ and level $\gN$ such that 
$$L(A, s) = \prod_{\tau: K\hookrightarrow \C} L(f^\tau, s).$$
For more background on Hilbert modular forms, see~\cite{dv13, hid88, shi78}.

In this section, we prove the following theorem:

\begin{thm}\label{thm:egr} Let $F$ be a real quadratic field of narrow class number one and discriminant $D \le 2000$.
Let $A$ be a modular abelian surface of $\GL_2$-type defined over $F$, with everywhere good reduction.
Then, except for $D = 353, 421, 1321, 1597$ or $1997$, we have one of the following:
\begin{enumerate}[(i)]
\item A is an $F$-surface, i.e. there is an abelian fourfold $B$ of $\GL_2$-type defined over $\Q$ such that $B \times_\Q F$ is isogenous to $A \times {}^\sigma\!A$,
where $\Gal(F/\Q) = \langle \sigma \rangle$; or
\item There is an abelian surface $B$ defined over $\Q$ such that $\End_\Q(B) = \Z$ and $B \times_\Q F$ is isogenous to $A$.
\end{enumerate}
\end{thm}

\begin{proof} In Table~\ref{table:discs}, we have listed all the discriminants $D \le 2000$ where $F = \Q(\sqrt{D})$ has narrow
class number one, and there is a newform $f$ of weight $2$ and level $(1)$ over $F$ whose coefficient field is the quadratic field
$K$ of discriminant $D'$. The notation ${D'}^{(2)}$ means that $f$ and its $\Gal(F/\Q)$-conjugate ${}^\sigma\!f$ are not in the same 
Hecke constituent. (This table was computed using the Hilbert Modular Forms Package in \verb|magma|~\cite{magma}.) 
By assumption, if $A$ is a surface satisfying the conditions of Theorem~\ref{thm:egr}, then $A$ is defined over some 
$F = \Q(\sqrt{D})$ for some $D$ in Table~\ref{table:discs}, and has RM by one of the associated $D'$.

Let $D$ be such a discriminant, and $f$ a newform over $F = \Q(\sqrt{D})$ with coefficients in $K = \Q(\sqrt{D'})$. Except for
$D = 353, 421, 1321, 1597$ or $1997$, the Hecke constituent of $f$ is unique in its $\Gal(F/\Q)$-orbit. 
So $f$ satisfies one of the following conditions:
\begin{enumerate}[(i)]
\item $f$ is a base change from $\Q$, in which case 
$$a_{\sigma(\gp)}(f) = a_{\gp}(f),\,\,\text{for all primes} \,\,\gp.$$
\item ${}^\sigma\!f = f^\tau$, where $\Gal(K/\Q) = \langle \tau \rangle$, in which case
$$a_{\sigma(\gp)}(f) = \tau(a_{\gp}(f)),\,\,\text{for all primes} \,\,\gp.$$
\end{enumerate}

In Case (i), the form $f$ is a base change of a newform in $g \in S_2(D, (\frac{D}{\cdot}))$, the space of classical
forms of weight $2$ and level $\Gamma_1(D)$ with character $(\frac{D}{\cdot})$. Let $B_g$ be the fourfold associated 
to $g$ by the Eichler-Shimura construction~\cite{shi94}. From~\cite[\S\S 7.7]{shi94}, we have 
$$B_g\times_\Q F \sim A_f \times {}^\sigma\!A_f.$$
Hence $A_f$ is a base change (in the automorphic sense). In Case (ii), assume that there is an abelian surface $A_f$ 
attached to $f$. Then, by~\cite[Theorem 5.4]{cd17} (see also~\cite{dk16}), the isogeny class of $A_f$ descends to $\Q$. So, there exists a surface
$B$ defined over $\Q$, with $\End_{\Q}(B) = \Z$, such that $B\times_{\Q} F \sim A_f$. 
\end{proof}

\begin{table}\small
\caption{The discriminants $D$ for which there are Hilbert newforms of weight $(2,2)$
and level $(1)$ over $\Q(\sqrt{D})$, with quadratic coefficient field $\Q(\sqrt{D'})$.}
\label{table:discs}
\begin{tabular}{cccc}

{ \begin{tabular}{ >{$}c<{$} >{$}l<{$} }\toprule
D & \multicolumn{1}{c}{$D'$} \\
\midrule
53 & 8 \\ 
61 & 12 \\ 
73& 5 \\ 
193 & 17\\ 
233 & 17\\ 
277 & 29\\ 
349 & 21\\
{\bf 353} & {\bf 5^{(2)}}\\
\bottomrule
\end{tabular}

\smallskip } &

{ \begin{tabular}{ >{$}c<{$} >{$}l<{$} }\toprule
D & \multicolumn{1}{c}{$D'$} \\
\midrule
373 & 93\\
389 & 8\\
397 & 24\\
409 & 13\\
{\bf 421} & 5, {\bf 5^{(2)}}\\
433 & 12\\
461 & 29\\
613 & 21\\
\bottomrule
\end{tabular}

\smallskip } &

{\begin{tabular}{ >{$}c<{$} >{$}l<{$} }\toprule
D & \multicolumn{1}{c}{$D'$} \\
\midrule
677 & 13, 29, 85\\
709 & 5\\
797 & 8, 29\\
809 & 5\\
821 & 44\\
853 & 21\\
929 & 13\\
997 & 13\\
\bottomrule
\end{tabular}

\smallskip} &

\begin{tabular}{ >{$}c<{$} >{$}l<{$} }\toprule
D & \multicolumn{1}{c}{$D'$} \\
\midrule
1013 & 21,53\\
1109 & 5, 5, 53, 77\\
1277 & 5, 29\\
{\bf 1321} & {\bf 5^{(2)}}\\
1493 & 5, 65\\
{\bf 1597} & {\bf 5^{(2)}}\\
{\bf 1997} & {\bf 8^{(2)}}\\
&  \\
\bottomrule
\end{tabular}

\end{tabular}
\end{table}

\begin{rem}\rm There are several restrictions in Theorem~\ref{thm:egr} that are non-essential. For example, the assumption that
$F$ has narrow class number one can be removed given that the Hilbert Modular Forms Package in \verb|magma|~\cite{magma} 
can compute Hilbert modular forms without restriction on the class group. Also, it is possible to go well beyond our bound on the 
discriminant. However, our goal was to convey the general philosophy of our approach rather than doing extensive computations.
\end{rem}

\begin{rem}\rm In~\cite{dk16}, the authors give several methods for constructing the surfaces with satisfy the conditions of Theorem~\ref{thm:egr}. 
In particular, they found most of the surfaces for $D \le 1000$. However, they couldn't find the surfaces for the discriminants $D = 353$ and 
$421$, which are non-base change. The remaining sections are devoted to dealing with the exceptional discriminants listed in Theorem~\ref{thm:egr}. 
We found explicit equations for all of them, except for $D = 1321$ and $1997$. (See Remark~\ref{rem:missing-exes} for a discussion on the 
the missing examples.)
\end{rem}

\section{\bf The abelian surface for the discriminant $D = 353$}\label{sec:sqrt353}

 
\subsection{The field of $2$-torsion} Let $f$ be the newform listed in Table~\ref{table:frob-data353}. We recall that
the $\Gal(F/\Q)$-conjugate ${}^\sigma\!f$ is determined by the relation
$$a_{\gp}({}^\sigma\!f) = a_{\sigma(\gp)}(f),$$
for $\gp\subset \CO_F$ prime. From this and the Hecke eigenvalues in Table~\ref{table:frob-data353}, it is easy to see that 
$f$ and ${}^\sigma\!f$ are not in the same Hecke constituent. Assume that there is an abelian surface $A_f$ attached to $f$; so 
that the Galois conjugate ${}^\sigma\!A_f$ is attached to ${}^\sigma\!f$. Let $\rho_{f, 2}:\,\Gal(\Qbar/F) \to \GL_2(\Qbar_2)$ be
the $2$-adic Galois representation attached to $f$, and $\rho_{A_f, 2}:\,\Gal(\Qbar/F) \to \GL_2(\Qbar_2)$
be the $2$-adic representation in the Tate module of $A_f$. By construction, we have that $\rho_{A_f, 2} \simeq \rho_{f,2}$. So
reducing modulo $2$, we have $\bar{\rho}_{A_f, 2} \simeq \bar{\rho}_{f, 2}$. Preliminary computations using the Chebotarev density 
theorem suggest that the image of $\bar{\rho}_{f,2}$ is $S_3$. But we cannot certify this given that the analogues of the Sturm bound 
would be impractical  in this case. That motivates the following lemma.


\begin{lem}\label{lem:2-tors-fld-sqrt353} Let $K = F(A_f[2])$ be the field of $2$-torsion of the abelian surface $A_f$, and $L/\Q$ is normal closure. 
If $L$ is a solvable extension, then it is the splitting field of the polynomial $h = x^6 - 2x^5 + x^4 + 19x^3 - 19x^2 + 2$; 
and we have 
$$\Gal(L/\Q) = \Gal(K/F)^2\rtimes \Z/2\Z \simeq S_3^2 \rtimes \Z/2\Z.$$ 
\end{lem}

\begin{proof} We keep the above notations. Form Table~\ref{table:frob-data353}, we see that the ring of integers
of the coefficient $K_f = \Q(\sqrt{5})$ is $\CO_f = \Z[e]$, where $e = \frac{1+\sqrt{5}}{2}$. So we have $\rho_{f,2}: \Gal(\Qbar/F) \to \GL_2(\Z_2[e])$, 
and its reduction $\bar{\rho}_{f, 2}: \Gal(\Qbar/F) \to \GL_2(\F_4)$. By the modularity assumption, we have $\bar{\rho}_{A_f, 2} = \bar{\rho}_{f, 2}$. 

We now compute a bound on the root discriminant $\delta_K$ of $K$. In Table~\ref{table:frob-data353}, we have listed
the Hecke eigenvalues $a_\gp(f) \bmod 2$ for all primes of norm up to $19$. For each of those primes, we have computed
$o_2(\gp)$ the order of the image of $\Frob_\gp$ modulo unipotents under the projectivisation $\mathrm{P}\bar{\rho}_{f,2}$
of $\bar{\rho}_{f,2}$.

Let $\gp = (9 + w)$ and $\sigma(\gp) = (10-w)$
be the two primes above $2$. From Table~\ref{table:frob-data353}, we see that 
$a_{\gp}(f) = -2e + 1 = 1 \in \F_4$ and $a_{\sigma(\gp)}(f) = \alpha \in \F_4$, where $\alpha^2 + \alpha + 1 = 0$. 
So $f$ is ordinary at $\gp$ and $\sigma(\gp)$. Hence, the mod $2$ representation restricted to the decomposition 
group $D_{\gp}$ at $\gp$ is of the form
\begin{align*}
\bar{\rho}_{f,2}|_{D_{\gp}} \simeq \begin{pmatrix} 1 & \ast \\ 0 & 1 \end{pmatrix} \!\mod 2,
\end{align*}
Similarly, the mod $2$ representation restricted to the decomposition group $D_{\sigma(\gp)}$ 
at $\sigma(\gp)$ is of the form
\begin{align*}
\bar{\rho}_{f,2}|_{D_{\sigma(\gp)}} \simeq \begin{pmatrix} \alpha & \ast \\ 0 & 1 \end{pmatrix} \!\mod 2.
\end{align*}
So, we can use the same argument as in \cite{dem09} to show that 
$\delta_K \le 4\cdot 353^{1/2} = 75.1531...$. (Note that this is the same as the Fontaine bound in Theorem~\ref{thm:fontaine-bds}.)

\begin{table}[t]\small
\caption{The Frobenius data for the two non Galois conjugate Hilbert newforms of weight $2$ and level $(1)$ over $F = \Q(\sqrt{353})$. 
(Here $e = \frac{1+\sqrt{5}}{2}$ and $\alpha$ is a cyclic generator of $\F_4^\times$.)}
\label{table:frob-data353}
\begin{tabular}{ >{$}c<{$} >{$}r<{$} >{$}c<{$} >{$}c<{$} >{$}c<{$} >{$}c<{$} >{$}c<{$}  } \toprule
\mathbf{N}\mathfrak{p} &\multicolumn{1}{c}{$\mathfrak{p}$} & a_\gp(f) &a_{\mathfrak{p}}(f)\bmod\,2 & o_2(\gp) & a_{\mathfrak{p}}(f)\bmod\,\sqrt{5} & o_{\sqrt{5}}(\gp) \\\midrule
2 & -w - 9 & 2e - 1 & 1 & - & 0 & 2\\
2 & w - 10 & -e + 1 & \alpha & - & 3 & 4\\
9 & 3 & -2e - 2 & 0 & 1 & 2 & 3\\
11 & -10w + 99 & 2e + 3 & 1 & 3 & 4 & 3\\
11 & 10w + 89 & -2e + 2 & 0 & 1 & 1 & 3\\
17 & -66w + 653 & -4e + 2 & 0 & 1 & 0 & 2\\
17 & -66w - 587 & 3 & 1 & 3 & 3 & 4\\
19 & -28w + 277 & 2 & 0 & 1 & 2 & 3\\
19 & -28w - 249 & 2e - 3 & 1 & 3 & 3 & 3\\
\bottomrule
\end{tabular}
\end{table}

The Galois extension $K/F$ is unramified outside $\gp$ and $\sigma(\gp)$. Assuming that
this extension is solvable, the Frobenius data shows that $\Gal(K/F)$ is either $C_3$, $S_3$ or $A_4$.
Since the $N = C_2 \times C_2$ is the only non trivial proper normal subgroup of $A_4$, the latter is only possible if $F$ admits
a cyclic cubic extension. However, since $F$ is real quadratic, and has class number one, it cannot have a cyclic cubic 
extension whose conductor is supported at $\gp$ and $\sigma(\gp)$ only. This excludes both $C_3$ and $A_4$. 
So $\Gal(K/F) = S_3$, and it must contain a quadratic extension $E'/F$ ramified at $\gp$ and $\sigma(\gp)$ only. The field $E'$ is given by 
an element in $\CO_F^\times/(\CO_F^\times)^2 =\{-1, \epsilon\}$, where $\epsilon$ is the fundamental unit in $\Z[\frac{1+\sqrt{353}}{2}]$. 

The extension $E'=F(\sqrt{-1}) = \Q(\sqrt{353},\sqrt{-1})$ is totally complex and unramified outside $\gp$ and $\sigma(\gp)$,  with $\Cl(E') = \Z/8\Z$.
There are no cyclic cubic extension of $E'$ unramified outside the primes above $\gp$ and $\sigma(\gp)$. So $\Q(\sqrt{353},\sqrt{-1})$ cannot be
the quadratic subfield of $K$. Therefore, we must have $E'= F(\sqrt{\epsilon})$.

The  field $E'=F(\sqrt{\epsilon})$ has 2 real places, and one complex place. It is unramified outside $\gp$ and $\sigma(\gp)$; and we have
$\Cl(E') = \Z/3\Z$. Letting $H_{E'}$ be the Hilbert class field of $E'$, we see that $H_{E'}$ is the only possible cubic extension of $E'$.
Further, a direct calculation shows that its Frobenius data matches that of the form $f$ listed in Table~\ref{table:frob-data353}. 
From this we obtain that $K = H_{E'}$ is the splitting field of the polynomial $g = x^3 - x^2 - w + 10 \in F[x]$. Its normal closure $L$ is given
by the polynomial $h = x^6 - 2x^5 + x^4 + 19x^3 - 19x^2 + 2 \in \Q[x]$, with $\Gal(L/\Q) \simeq S_3^2 \ltimes \Z/2\Z$. 
Since $L$ is solvable, we can compute its root discriminant using local class field theory, which gives
that $\delta_L = 2\cdot 353^{1/2} = 37.5765...$. 

Alternatively, we can use the Jones-Roberts Tables~\cite{jr14} to find the field $L$. Indeed, assuming that $L$ is solvable, it will be given
by a polynomial of degree $6$ or $8$. In the latter case, we have $\Gal(K/F) = A_4$. The tables are proven to be complete for all 
solvable polynomials of degree $6$ and $8$ such that the root discriminant of the normal closure is less than $75.1531...$. Only the
polynomial $h$ listed in Lemma~\ref{lem:2-tors-fld-sqrt353} matches the Frobenius data of the form $f$ given in Table~\ref{table:frob-data353}.

\end{proof}

\subsection{The search method}\label{subsec:search353}
Let us assume that there is an abelian surface $A = A_f$ associated to the Hecke constituent of the form $f$
of level $(1)$ and weight $2$ over $F = \Q(\sqrt{353})$ in Table~\ref{table:frob-data353}. Then, the surface $A$ has RM by $\Z[\frac{1+\sqrt{5}}{2}]$ where 
$\frac{1+\sqrt{5}}{2}$ is a unit of norm $-1$ in $\Q(\sqrt{5})$. Therefore, by~\cite[Proposition 3.11]{ggr05}, $A$ is principally
polarisable. Let $C$ be a genus $2$ curve defined over $F$ such that $A = \Jac(C)$. Then, there is a curve $C':\,y^2 = h'(x)$ where 
$h'(x) \in F[x]$ has degree $5$ or $6$, such that $A' = \Jac(C')$ is isomorphic to $A$ over $F$. By using the Hecke eigenvalues 
$a_\gp = a_\gp(f)$ in Table~\ref{table:frob-data353}, we obtain that
\begin{align*}
\#A(\F_\gp) = \mathrm{N}_{\Q(\sqrt{5})/\Q}(\mathrm{N}_{F/\Q}(\gp) + 1 - a_\gp) = \mathrm{N}_{\Q(\sqrt{5})/\Q}(2 + 1 - e) = 5,
\end{align*} for the prime $\gp = (w - 10)$ above $2$. Since $A(F)_{tors}$ injects into $A(\F_\gp)$, this implies that $A$ 
does not have a point of order $2$ defined over $F$. By combining
this with Lemma~\ref{lem:2-tors-fld-sqrt353} and Theorem~\ref{thm:2-torsion-grps}, we see that the polynomial $h'$ is of the form 
$h' = h_{\alpha} h_{\alpha'}$ where $K= F[c]$ is the cubic extension defined by the cubic factor $g := x^3 - x^2 - w + 10$ of $h$, 
and $h_\alpha, h_{\alpha'}$ are the minimal polynomials of some elements $\alpha, \alpha' \in K \setminus F$.

By making a search over the integral elements in $K$ using Algorithm~\ref{alg:bd-height} described in Section~\ref{sec:heights},
we obtain the pair $(\alpha, \alpha')$ with
\begin{align*}
\alpha := \frac{-22c^5 + 35c^4 - 10c^3 - 419c^2 + 235c + 83}{17},\\
\alpha' := \frac{-36c^5 + 48c^4 - 4c^3 - 698c^2 + 264c + 74}{17},
\end{align*} with $H_K(\alpha) = 64.0000...$ and $H_K(\alpha') = 1856.3958...$. This gives the polynomial
\begin{align*}
h'(x) &:= x^6 + 8x^5 + (w - 64)x^4 + (20w - 80)x^3 + (-44w + 240)x^2\\
&\qquad{} + (-96w + 1088)x - 64w + 576.
\end{align*}
From this, wee obtain the global minimal model $C$ displayed in Theorem~\ref{thm:sqrt353}.

\subsection{The surfaces}

\begin{thm}\label{thm:sqrt353}
Let $F = \Q(\sqrt{353})$, and $w =\frac{1+\sqrt{353}}{2}$, and define the curve $C:\,y^2 + Q(x)y = P(x)$ by
\begin{align*}
P(x) &:= -(15w + 149)x^6 - (1119w + 9948)x^5 - (36545w + 325090)x^4\\
&\qquad - (636332w + 5659370)x^3 - (6227174w + 55387985)x^2\\
&\qquad - (32480001w + 288869715)x - 70532813w - 627353458;\\
Q(x) &:= (w + 1)x^3 + x^2 + wx + w + 1.
\end{align*}
Let ${}^\sigma\!C$ denote the Galois conjugate of $C$, $A$ and ${}^\sigma\!A$ the Jacobians of $C$  and ${}^\sigma\!C$ respectively. 
Then, we have the followings:
\begin{enumerate}[(a)]
\item The discriminant of the curve $C$ is $\mathrm{disc}(C) = -\epsilon^4$, where $\epsilon$ is the fundamental unit. So $C$, ${}^\sigma\!C$ and the
surfaces $A$, ${}^\sigma\!A$ have everywhere good reduction.
\item $A$ and ${}^\sigma\!A$ have real multiplication by $\Z[\frac{1+\sqrt{5}}{2}]$.   
\item  $A$ and ${}^\sigma\!A$ are modular. They correspond to the two Hecke constituents of $f, {}^\sigma\!f \in S_2(1)$ of dimension $2$, where $S_2(1)$
is the space of Hilbert modular forms of level $(1)$ and weight $2$ over $F$ (see Table~\ref{table:frob-data353}). 
\item $A$ and ${}^\sigma\!A$ are non-isogenous.
\end{enumerate}
\end{thm}

\begin{proof}
(a) This is just an easy calculation.

(b) It is enough to prove this for the surface $A$. Using the equation of the Humbert surface for the discriminant $D=5$ given in~\cite{ek14}, 
we find that $A$ is a twist of the surface corresponding to the point 
\begin{align*}
(g,h) = \left(-\frac{12w + 209}{726},\frac{742w + 6611}{161051}\right).
\end{align*}

(c) From (b), we know that $A$ and ${}^\sigma\!A$ are of $\GL_2$-type. In Table~\ref{table:frob-data353}, we have listed
the Hecke eigenvalues $a_\gp(f) \bmod \sqrt{5}$ for all primes of norm up to $19$. For each of those primes, we have computed
$o_{\sqrt{5}}(\gp)$ the order of the image of $\Frob_\gp$ modulo unipotents under the projectivisation $\mathrm{P}\bar{\rho}_{f,\sqrt{5}}$
of $\bar{\rho}_{f,\sqrt{5}}$. From the orders of the elements, it follows that the projective image of 
$\bar{\rho}_{f, \sqrt{5}}$ is either $S_4$ or $\PGL(2,\F_5)$. By computing the Euler factors of $A$ then factoring over $\Q(\sqrt{5})$, we 
check that for each prime $\gp$ in that table, we have 
$$\Tr(\rho_{A,\sqrt{5}}(\Frob_\gp)) = a_\gp(f)\,\,\text{or}\,\, \tau(a_\gp(f)),$$ where $\Gal(\Q(\sqrt{5})/\Q) = \langle \tau\rangle$. 
Up to Galois conjugation, these agree with those of the form $f$. 
So, the
projective image of $\bar{\rho}_{A, \sqrt{5}}$ is either $S_4$ or $\PGL(2,\F_5)$. Hence $\bar{\rho}_{A, \sqrt{5}}$ cannot be dihedral. 

By \cite[Theorem 1.2]{sbt97}, there is an elliptic curve $E/F$ such that $\bar{\rho}_{A, \sqrt{5}} \simeq \bar{\rho}_{E,5}$. 
The curve $E$ is modular by~\cite[Theorem 1]{fls15}. Thus $\bar{\rho}_{A, \sqrt{5}}$ is modular. Since $A$ is an abelian surface, 
it is clear that $\rho_{A,\sqrt{5}}$ satisfies the remaining hypotheses of~\cite[Theorem 1.1]{kt17}. So, we conclude that $\rho_{A,\sqrt{5}}$,
and hence $A$, is modular. 

(d) The surfaces $A$ and ${}^\sigma\!A$ correspond to different Hecke constituents. Therefore, they have different isogeny classes by 
Faltings~\cite[Korollar 2]{fal83}.
\end{proof}

\begin{rem}\rm 
The field $F = \Q(\sqrt{353})$ appears to be the real quadratic field of narrow class number one, with the smallest discriminant such that there 
is an abelian surface with RM  and everywhere good reduction that is non-isogenous to its Galois conjugate. In that sense, this example would be 
the analogue in dimension $2$ of the elliptic curve of conductor $(1)$ over $\Q(\sqrt{509})$ discovered by Pinch~\cite{pin82}. However, we cannot prove this 
without assuming modularity. 
\end{rem}

\section{\bf The abelian surface for the discriminant $D = 421$}\label{sec:sqrt421}

\subsection{The field of $2$-torsion} 
In this example, the defining polynomial for the $2$-torsion field cannot be obtained via class field theory. Indeed, an inspection of the
Frobenius data for the mod $3$ representation leads us to the following lemma. 

\begin{lem}\label{lem:2-tors-fld-sqrt421} Assume that there is an abelian surface $A_f$ attached to the form $f$ listed in Table~\ref{table:frob-data421}.
Let $K = F(A_f[2])$ be the field of $2$-torsion of $A_f$, and $L/\Q$ the normal closure of $K$. 
Then $L$ is unramified outside $2$ and $421$, with Galois group $$\Gal(L/\Q) \simeq A_5^2 \rtimes \Z/2\Z,$$
and we have $\delta_L < 82.0731...$.
\end{lem}

\begin{proof} In Table~\ref{table:frob-data421}, we have listed the Hecke eigenvalues $a_\gp(f) \bmod 2$ for all primes of norm up to $11$. 
For each of those primes, we have computed $o_2(\gp)$ the order of the image of $\Frob_\gp$ modulo unipotents under the projectivisation 
$\mathrm{P}\bar{\rho}_{f,2}$ of $\bar{\rho}_{f,2}$. The first part of the lemma follows by inspection of that Frobenius data. The second part 
concerning the root discriminant uses the same argument as in~\cite{dem09}.
\end{proof}

\begin{table}[t]\small
\caption{The Frobenius data for the two non Galois conjugate Hilbert newforms of weight $2$ and level $(1)$ over $F = \Q(\sqrt{421})$. 
(Here $e = \frac{1 + \sqrt{5}}{2}$ and $\alpha$ is a cyclic generator of $\F_4^\times$.) }
\label{table:frob-data421}
\begin{tabular}{ >{$}c<{$} >{$}r<{$} >{$}c<{$} >{$}c<{$} >{$}c<{$} >{$}c<{$} >{$}c<{$}  } \toprule
\mathbf{N}\mathfrak{p} &\multicolumn{1}{c}{$\mathfrak{p}$} & a_\gp(f) &a_{\mathfrak{p}}(f)\bmod\,2 & o_2(\gp) & a_{\mathfrak{p}}(f)\bmod\,\sqrt{5} & o_{\sqrt{5}}(\gp) \\\midrule
3 & 4w - 43 & -2e + 1 & 1 & 3 & 0 & 2\\
3 & -4w - 39 & 2e & 0 & 1 & 1 & 4\\
4 & 2 & e - 2 & \alpha^2 & - & 1 & 1\\
5 & w - 11 & e - 2 & \alpha^2 & 5 & 1 & -\\
5 & -w - 10 & 3 & 1 & 3 & 3 & -\\
7 & -54w - 527 & 3 & 1 & 3 & 3 & 4\\
7 & 54w - 581 & e - 2 & \alpha^2 & 5 & 1 & 6\\
11 & 25w + 244 & 0 & 0 & 1 & 0 & 2\\
11 & -25w + 269 & -e + 5 & \alpha & 5 & 2 & 1\\ 
\bottomrule
\end{tabular}
\end{table}

The following polynomial was obtained by an extensive search:
\begin{align*}
h&= x^{12} - 18x^{10} + 26x^9 + 58x^8 - 212x^7 - 40x^6 + 766x^5 + 268x^4 - 1030x^3\\
&\qquad{} - 989x^2 - 366x - 99.
\end{align*}
It was kindly provided by Eric Driver, John Jones and David P. Roberts as a potential candidate for the
defining polynomial for  the field $L$ in Lemma~\ref{lem:2-tors-fld-sqrt421}. Theorem~\ref{thm:sqrt421} 
below confirms that their prediction was accurate, although the search they did was not exhaustive.

\subsection{The search method}\label{subsec:search421}
Let us assume that there is an abelian surface $A = A_f$ associated to the Hecke constituent of the form $f$
of level $1$ and weight $2$ over $F = \Q(\sqrt{421})$ in Table~\ref{table:frob-data421}. Then, the surface $A$ has RM by $\Z[\frac{1+\sqrt{5}}{2}]$ where 
$\frac{1+\sqrt{5}}{2}$ is a unit of norm $-1$ in $\Q(\sqrt{5})$. Therefore, by~\cite[Proposition 3.11]{ggr05}, $A$ is principally
polarisable. Let $C$ be a genus $2$ curve defined over $F$ such that $A = \Jac(C)$. Then, there is a curve $C':\,y^2 = h'(x)$ where 
$h'(x) \in F[x]$ has degree $5$ or $6$, such that $A' = \Jac(C')$ is isomorphic to $A$ over $F$. By using the Hecke eigenvalues in 
Table~\ref{table:frob-data421}, we obtain that
\begin{align*}
\#A(\F_\gp) = \mathrm{N}_{\Q(\sqrt{5})/\Q}(\mathrm{N}_{F/\Q}(\gp) + 1 - a_\gp) = \mathrm{N}_{\Q(\sqrt{5})/\Q}(5 + 1 - 3) = 9,
\end{align*} for the prime $\gp = (-w - 10)$ above $5$. Hence $A$ does not have a point of order $2$ defined over $F$. By combining
this with Lemma~\ref{lem:2-tors-fld-sqrt421} and Theorem~\ref{thm:2-torsion-grps}, we see that the polynomial $h'$ is of the form 
$h' = h_{\alpha}$ where $K= F[c]$ is the sextic extension defined by 
$g := x^6 - 9x^4 + (-2w + 14)x^3 + (3w - 13)x^2 + (2w + 10)x + w + 2$, and $h_\alpha$
the minimal polynomial of some element $\alpha \in K \setminus F$.

By making a search over the integral elements in $K$ using a variant of the Algorithm~\ref{alg:bd-height} described in Section~\ref{sec:heights},
we obtain an element $\alpha$ with $H_K(\alpha) = 416759.3936...$ which we do not display here. By clearing denominators, we this gives the 
polynomial with integral coefficients 
\begin{align*}
h'(x) &:= x^6 + (-824w + 6396)x^5 + (-5950152w + 15262668)x^4\\
&\qquad\quad{} + (15357307104w + 189762599664)x^3\\
&\qquad\quad{} + (-200691458540784w + 1196268593295456)x^2\\
&\qquad\quad{} + (225275530789117440w - 1153351434605863104)x\\
&\qquad\quad{} + 886640916155668875072w - 10173976221331135198656.
\end{align*}
From this, wee obtain a global minimal for $C$ displayed in Theorem~\ref{thm:sqrt421}.

\subsection{The surfaces}
\begin{thm}\label{thm:sqrt421} Let $F = \Q(\sqrt{421})$, and $w =\frac{1+\sqrt{421}}{2}$, and define the curve $C:\,y^2 + Q(x)y = P(x)$ by
\begin{align*}
P &:= (13w + 77)x^6 + (593w + 6772)x^5 + (15049w + 131460)x^4\\
&\qquad + (163829w + 1727293)x^3 + (1167345w + 10787410)x^2\\
&\qquad + (3985370w + 40412781)x + 6111237w + 58050373;\\
Q &:= wx^3 + x^2 + (w + 1)x + w + 1.
\end{align*}
Let ${}^\sigma\!C$ denote the Galois conjugate of $C$, $A$ and ${}^\sigma\!A$ the Jacobians of $C$  and ${}^\sigma\!C$ respectively. 
Then, we have the following:
\begin{enumerate}[(a)]
\item The discriminant of the curve $C$ is $\mathrm{disc}(C) = \epsilon^{8}(25w + 244)^{22}$, where $\epsilon$ is the fundamental unit and $(25w + 244)$
is one of the primes above $11$. The surfaces $A$ and ${}^\sigma\!A$ have everywhere good reduction. 
\item $A$ and ${}^\sigma\!A$ have real multiplication by $\Z[\frac{1+\sqrt{5}}{2}]$.   
\item  $A$ and ${}^\sigma\!A$ are modular. They correspond to the two Hecke constituents of $f, {}^\sigma\!f \in S_2(1)$ of dimension $2$, where $S_2(1)$
is the space of Hilbert modular forms of level $(1)$ and weight $2$ over $F$ (see Table~\ref{table:frob-data421}).  
\item $A$ and ${}^\sigma\!A$ are non-isogenous.
\end{enumerate}
\end{thm}

\begin{proof} Again (a) is an easy calculation; in this case, one shows that the conductor of $A$ at the prime $(25w + 244)$ is trivial. 
To prove (b), we use the equation of the Humbert surface for the discriminant $D=5$ in~\cite{ek14}.
We find that $A$ is a twist of the surface corresponding to the point
\begin{align*}
(g,h)= \left(\frac{-16w + 147}{54}, \frac{-12229w + 129846}{243}\right). 
\end{align*}
In Table~\ref{table:frob-data421}, we have listed
the Hecke eigenvalues $a_\gp(f) \bmod \sqrt{5}$ for all primes of norm up to $11$. For each of those primes, we have computed
$o_{\sqrt{5}}(\gp)$ the order of the image of $\Frob_\gp$ modulo unipotents under the projectivisation $\mathrm{P}\bar{\rho}_{f,\sqrt{5}}$
of $\bar{\rho}_{f,\sqrt{5}}$. From those orders, we see that the projective image of $\bar{\rho}_{f, \sqrt{5}}$ is $\PGL_2(\F_5)$. A similar argument
as in the proof of Theorem~\ref{thm:sqrt353} shows that $\bar{\rho}_{A,\sqrt{5}}$ is modular and surjective. So, we conclude (c) by using~\cite[Theorem 1.1]{kt17}.
Finally, the surfaces $A$ and ${}^\sigma\!A$ are non-isogenous since the forms $f$ and ${}^\sigma\!f$ are not in the same Hecke constituent. This concludes (d).
\end{proof}

\begin{rem}\label{rem:search421}\rm We note that the surface $A$ in Theorem~\ref{thm:sqrt421} has good reduction at the prime $\gp = (25w + 244)$ above $11$
even though the curve $C$ has bad reduction at that prime. In this case, the reduction of $A$ is isomorphic to the product of two elliptic curves. 
This forces the Euler factor of $A$ at $\gp$ to be the square of that of an elliptic curve. For this reason, the corresponding Hecke eigenvalue will be
an integer. Here, we have $a_{\gp} = 0$.

From this discussion, we see that the set of primes of bad reduction for the curve, which are primes of good reduction for the Jacobian, is contained
the set of primes where the Hecke eigenvalues are integers. Unfortunately, the latter set is infinite, and cannot be used to bound the former. 
This makes the search for the curve $C$ above much trickier. 

\end{rem}

\section{\bf The abelian surface for the discriminant $D = 1597$}\label{sec:sqrt1597}

\subsection{The field of $2$-torsion} In this case, the knowledge of the Frobenius data for the form (see Table~\ref{table:frob-data1597})
allows us to prove the following lemma. But we were unable to obtain the field of $2$-torsion via a search. 

\begin{lem} Assume that there is an abelian surface $A_f$ attached to the form $f$ listed in Table~\ref{table:frob-data1597}.
Let $K = F(A_f[2])$ be the field of $2$-torsion of $A_f$, and $L/\Q$ the normal closure of $K$. Then $L$ is unramified outside 
$2$ and $1597$, with Galois group $$\Gal(L/\Q) \simeq A_5^2 \rtimes \Z/2\Z,$$ and we have $\delta_L < 159.8499...$.
\end{lem}

\begin{proof} In Table~\ref{table:frob-data1597}, we have listed the Hecke eigenvalues $a_\gp(f) \bmod 2$ for all primes of norm up to $19$. 
For each of those primes, we have computed $o_2(\gp)$ the order of the image of $\Frob_\gp$ modulo unipotents under the projectivisation 
$\mathrm{P}\bar{\rho}_{f,2}$ of $\bar{\rho}_{f,2}$. The proof of the lemma follows that of Lemma~\ref{lem:2-tors-fld-sqrt421}.
\end{proof}

\begin{table}[t]\small
\caption{The Frobenius data for the two non Galois conjugate Hilbert newforms of weight $2$ and level $(1)$ over $F = \Q(\sqrt{1597})$. 
(Here $e = \frac{1+\sqrt{5}}{2}$ and $\alpha$ is a cyclic generator of $\F_4^\times$.) }
\label{table:frob-data1597}
\begin{tabular}{ >{$}c<{$} >{$}r<{$} >{$}c<{$} >{$}c<{$} >{$}c<{$} >{$}c<{$} >{$}c<{$}  } \toprule
\mathbf{N}\mathfrak{p} &\multicolumn{1}{c}{$\mathfrak{p}$} & a_\gp(f) &a_{\mathfrak{p}}(f)\bmod\,2 & o_2(\gp) & a_{\mathfrak{p}}(f)\bmod\,\sqrt{5} & o_{\sqrt{5}}(\gp) \\\midrule
3 & -2w - 39 & 2e & 0 & 1 & 1 & 4\\
3 & -2w + 41 & -e + 2 & \alpha^2 & 5 & 4 & 4\\
4 & 2 & e - 2 & \alpha^2 & - & 1 & 1\\
7 & 27w - 553 & 3e & \alpha^2 & 5 & 4 & 6\\
7 & 27w + 526 & 2e - 3 & 1 & 3 & 3 & 4\\
17 & -773w + 15832 & 6 & 0 & 1 & 1 & 6\\
17 & 773w + 15059 & 5e - 1 & \alpha & 5 & 4 & 6\\
19 & -w - 19 & -2e + 5 & 1 & 3 & 4 & 1\\
19 & -w + 20 & 3e - 2 & \alpha^2 & 5 & 2 & 3\\
\bottomrule
\end{tabular}
\end{table}

\subsection{The search method}\label{subsec:search1597}
Let us assume that there is an abelian surface $A = A_f$ associated to the Hecke constituent of the form $f$
of level $(1)$ and weight $2$ over $F = \Q(\sqrt{1597})$ in Table~\ref{table:frob-data1597}. Then, the surface $A$ has RM by $\Z[\frac{1+\sqrt{5}}{2}]$ where 
$\frac{1+\sqrt{5}}{2}$ is a unit of norm $-1$ in $\Q(\sqrt{5})$. Therefore, by~\cite[Proposition 3.11]{ggr05}, $A$ is principally
polarisable. Let $C$ be a genus $2$ curve defined over $F$ such that $A = \Jac(C)$. Then, there is a curve $C':\,y^2 = h'(x)$ where 
$h'(x) \in F[x]$ has degree $5$ or $6$, such that $A' = \Jac(C')$ is isomorphic to $A$ over $F$.

In this case, we use the search method described in~\cite{dk16}. Let $Y_{-}(5)$ be the Hilbert modular surface of discriminant $5$ which parametrises all
principally polarised abelian surfaces with RM by $\Z[\frac{1+\sqrt{5}}{2}]$. In~\cite[Theorem 16]{ek14}, the surface $Y_{-}(5)$ is described as a
double cover of the weighted projective space $\mathbf{P}_{g, h}^2$:
$$z = 2(6250h^2-4500g^2h-1350gh-108h-972g^5-324g^4-27g^3).$$
For the abelian surface $A$ attached to the rational point $(g, h)$, the Igusa-Clebsch invariants of $A$ are given by~\cite[Corollary 15]{ek14}. In this case, we
are looking for a surface $A = \Jac(C)$ such that the discriminant of $C$ is a scalar multiple of $h^2$. However, the fundamental unit 
$\epsilon = -2518525w - 49063993$ has height $100646511.0000...$. So a naive height search as in~\cite{dk16} will not work. So, we scale
the parameters by setting
$$(g, h) = \left(\frac{\epsilon^m g'}{6u^2}, \frac{\epsilon^{n}h'}{u^5}\right),$$
where $m, n \in \{0, \pm 1, \pm 2, \pm 3, \pm 4\}$ and $g', h', u$ are integral elements with small height. Letting $m = 1, n = 2$, $g' = 114w - 2335, h' = 1$ and $u = -27w - 526$, we get
$$(g, h) = \left(\frac{16w + 259}{294}, \frac{2913w + 56749}{16807} \right).$$
We get the curve $C'$ using the same approach as in~\cite{dk16}. By reduction, this yields the global model for $C$ displayed in Theorem~\ref{thm:sqrt1597}.

\begin{rem}\label{rem:search1597}\rm
The scaling trick introduced in the search method in Subsection~\ref{subsec:search1597} was fine tuned using the curves we found
for the discriminants $D = 353$ and $421$. In this case, the trick was successful because the curve $C$ has a unit discriminant. 
In general, the scaling must take into account the set of primes of bad reduction for $C$ that are primes of good reduction for its 
Jacobian $A$. However, as indicated in Remark~\ref{rem:search421}, it is not possible to predict those primes even though we 
know they are contained in the set of primes where the Hilbert newform associated to $A$ has integer Hecke eigenvalues. 
Therefore it is very difficult to use this trick to find curves such as the one in Theorem~\ref{thm:sqrt421}.
\end{rem}

\subsection{The surfaces}

\begin{thm}\label{thm:sqrt1597} Let $F = \Q(\sqrt{1597})$, and $w =\frac{1+\sqrt{1597}}{2}$, and define the curve $C:\,y^2 + Q(x)y = P(x)$ by
\begin{align*}
P &:= (14154412w + 275745514)x^6 - (489014393w + 9526607332)x^5\\
&\quad{} + (7039395048w + 137136152764)x^4 - (54043428224w + 1052833060832)x^3\\
&\quad{} + (233382395752w + 4546578743807)x^2 - (537510739916w + 10471376373574)x\\
&\quad{} + 515810377784w + 10048626384323;\\
Q &:= x^3 + wx^2 + (w + 1)x + w + 1
\end{align*}

Let ${}^\sigma\!C$ denote the Galois conjugate of $C$, $A$ and ${}^\sigma\!A$ the Jacobians of $C$  and ${}^\sigma\!C$ respectively. 
Then, we have the followings:
\begin{enumerate}[(a)]
\item The discriminant of the curve $C$ is $\mathrm{disc}(C) = \bar{\epsilon}^{6}$, where $\epsilon$ is the fundamental unit. So $C$, ${}^\sigma\!C$ and the
surfaces $A$, ${}^\sigma\!A$ have everywhere good reduction.
\item $A$ and ${}^\sigma\!A$ have real multiplication by $\Z[\frac{1+\sqrt{5}}{2}]$.   
\item  $A$ and ${}^\sigma\!A$ are modular. They correspond to the two Hecke constituents of $f, {}^\sigma\! f \in S_2(1)$ of dimension $2$, 
the space of Hilbert modular forms of level $(1)$ and weight $2$ over $F$ (see Table~\ref{table:frob-data1597}). 
\item $A$ and ${}^\sigma\!A$ are non-isogenous.
\end{enumerate}
\end{thm}

\begin{proof} During the search we described above, we showed that $A$ is a twist of the surface corresponding to the point
\begin{align*}
(g, h) = \left(\frac{16w + 259}{294}, \frac{2913w + 56749}{16807} \right).
\end{align*}
This shows that $A$ has RM by $\Z[\frac{1+ \sqrt{5}}{2}]$. In Table~\ref{table:frob-data1597}, we have listed the Hecke eigenvalues 
$a_\gp(f) \bmod \sqrt{5}$ for all primes of norm up to $19$. 
For each of those primes, we have computed $o_{\sqrt{5}}(\gp)$ the order of the image of $\Frob_\gp$ modulo unipotents under the 
projectivisation $\mathrm{P}\bar{\rho}_{f,\sqrt{5}}$ of $\bar{\rho}_{f,\sqrt{5}}$. From the orders, we see that the residual representation 
$\bar{\rho}_{A,\sqrt{5}}$ is surjective. So, the rest of the proof of the theorem follows as in Theorem~\ref{thm:sqrt421}.
\end{proof}

\begin{rem}\label{rem:missing-exes}\rm For the discriminants $D = 1321, 1997$, we could not find the corresponding surfaces. In both cases,
we strongly believe that the surfaces are Jacobians of curves whose minimal models have primes of bad reduction despite the surfaces 
themselves having everywhere good reduction. As explained in Remark~\ref{rem:search1597}, the scaling techniques introduced in 
Subsection~\ref{subsec:search1597} cannot be applied in those situations.
\end{rem}

\section{\bf The abelian surfaces for the discriminant $D = 1997$}\label{sec:sqrt1997}

\subsection{The field of $2$-torsion}

\begin{lem}\label{lem:sqrt2-tors-fld-sqrt1997} Assume that there is an abelian surface $A_f$ attached to $f$ Table~\ref{table:frob-data1997},
 and let $K$ be the Galois closure of the field $F(A_f[\sqrt{2}])$. 
Then $K$ is the splitting field of the polynomial $h' =x^6 - 25x^4 - 50x^3 - 343x^2 - 1372x - 1372$. We have $\Gal(K/\Q) \simeq S_3^2 \ltimes \Z/2\Z$, and 
$\delta_K = 2^{3/2}\cdot 1997^{1/2} = 126.396...$. 
\end{lem}

\begin{proof} Let $\rho_{f,\sqrt{2}}: \Gal(\Qbar/F) \to \GL_2(\Z_2[\sqrt{2}])$ be the $\sqrt{2}$-adic representation attached to $f$, and 
$\bar{\rho}_{f,\sqrt{2}}: \Gal(\Qbar/F) \to \GL_2(\F_2)$ its reduction modulo $(\sqrt{2})$. By assumption, we have 
$\bar{\rho}_{A_f, \sqrt{2}} = \bar{\rho}_{f, \sqrt{2}}$, where $\bar{\rho}_{A_f,\sqrt{2}}: \Gal(\Qbar/F) \to \GL(A_f[\sqrt{2}])$ the mod $\sqrt{2}$ 
attached to $A_f$. 

We now compute a bound on the root discriminant $\delta_K$ of $K$. In Table~\ref{table:frob-data1997}, we have listed
the Hecke eigenvalues $a_\gp(f) \bmod 2$ for all primes of norm up to $29$. For each of those primes, we have computed
$o_2(\gp)$ the order of the image of $\Frob_\gp$ modulo unipotents under the projectivisation $\mathrm{P}\bar{\rho}_{f,2}$
of $\bar{\rho}_{f,2}$.  From that data, we see that 
$a_{(2)}(f) = -1 - \sqrt{2} = 1\, (\bmod \sqrt{2})$. So $f$ is ordinary at $(2)$. Hence, the mod $\sqrt{2}$-representation
restricted to the decomposition group at $(2)$ is of the form
$$\bar{\rho}_{f,\sqrt{2}}|_{D_{(2)}} \simeq \begin{pmatrix} 1 & \ast \\ 0 & 1 \end{pmatrix} \,\mod \sqrt{2}.$$
So, we can use the same argument as in \cite{dem09} to show that 
$\delta_K \le 4\cdot 1997^{1/2} = 178.751...$. (This is the same as the Fontaine bound in Theorem~\ref{thm:fontaine-bds}.)

The field $F(A_f[\sqrt{2}])$ is a Galois extension of $F$ which is unramified outside the prime $(2)$. The
Frobenius data shows that $[F(A_f[\sqrt{2}]) : F]$ is either $3$ or $6$. Since $F$ has narrow class number one, $F(A_f[\sqrt{2}])$ 
is unramified outside $(2)$,  it cannot be a cyclic cubic extension.  So $[F(A_f[\sqrt{2}]) : F] = 6$, and it must contain a quadratic
extension $E'/F$ ramified at (2) only. The field $E'$ is given by an element in $\CO_F^\times/(\CO_F^\times)^2 =\{-1, \epsilon\}$,
where $\epsilon$ is the fundamental unit in $\Z[\frac{1+\sqrt{1997}}{2}]$. 

The extension $E'=F(\sqrt{-1})$ is a totally complex field in which the prime $(2)$ ramifies. It has no cubic extension
whose conductor is a power of the prime above $2$. However,  the class group of $\Cl(F(\sqrt{-1})) = \Z/21\Z = \langle \eta \rangle$. 
The Frobenius data of the form $f$ does not match that arising from the extension associated to $\eta^7$. So, we must have $E'= F(\sqrt{\epsilon})$. 

The  field $F(\sqrt{\epsilon})$ has 2 real places, and one complex place;  and the prime (2) ramifies. Again, there are no cubic extension of 
$F(\sqrt{\epsilon})$ whose conductor is a power of the prime above $2$. This means that $F(A_f[\sqrt{2}])/F(\sqrt{\epsilon})$ must be an 
unramified cubic extension. (Note that this is consistent with the fact that $[F(A_f[\sqrt{2}]) : F] = 6$.) 
The class group of $F(\sqrt{\epsilon})$ is cyclic of order $3$. So, its Hilbert class field is an unramified cubic extension of 
$F(\sqrt{\epsilon})$ generated by the polynomial $g'= x^6 + (-w - 3)x^4 + (2w + 51)x^2 - w + 15$ over $F$. So, we have that
$F(A_f[\sqrt{2}]) = F[x]/(g'(x))$, and its normal closure $K$ is given by the polynomial
$g = x^{12} - 7x^{10} - 383x^8 + 1662x^6 - 393x^4 + 3505x^2 - 289$, which is the norm of $g'$. There are three
subfields of degree $6$ in the field $K$, and they are all isomorphic. One of them
is given by the polynomial $h'$. So, by construction, we obtain that $K$ is the splitting field of $h'$. By explicit
calculations, one gets that $\Gal(K/\Q) \simeq S_3^2 \ltimes \Z/2\Z$. 
Since $K$ is solvable, we can compute its root discriminant using local class field theory, which gives
that $\delta_K = 2^{3/2} 1997^{1/2} = 126.396...$. 

Alternatively, we can also look it up in~\cite{jr14} using the Frobenius data in Table~\ref{table:frob-data1997}. 
There is a unique polynomial $h$ whose Frobenius data matches the one in the table, and such that the root discriminant of the splitting
field satisfies the Fontaine bound in Theorem~\ref{thm:fontaine-bds}. The tables are complete in this case. 
\end{proof}

\begin{table}[t]\small
\caption{The Frobenius data for the two non Galois conjugate Hilbert newforms of weight $2$ and level $(1)$ over $F = \Q(\sqrt{1997})$. 
(Here $e = \sqrt{2}$).}
\label{table:frob-data1997}
\begin{tabular}{ >{$}c<{$} >{$}r<{$} >{$}c<{$} >{$}c<{$} >{$}c<{$}  } \toprule
\mathbf{N}\mathfrak{p} &\multicolumn{1}{c}{$\mathfrak{p}$} & a_\gp(f) &a_{\mathfrak{p}}(f)\bmod\,2 & o_2(\gp) \\\midrule
4 & 2 & -e - 1 & 1 & -\\
7 & w + 22 & -e + 2 & 0 & 1\\
7 & -w + 23 & -3 & 1 & 3\\
9 & 3 & -e & 0 & 1\\
17 & -6w + 137 & -3e - 2 & 0 & 1\\
17 & -6w - 131 & 6 & 0 & 1\\
25 & 5 & 2e - 5 & 1 & 3\\
29 & 7w + 153 & -3e + 2 & 0 & 1\\
29 & -7w + 160 & -3e + 2 & 0 & 1\\
\bottomrule
\end{tabular}
\end{table}

\begin{thm}\label{thm:2-tors-fld-sqrt1997} 
Assume that there is an abelian surface $A_f$ over $F = \Q(\sqrt{1997})$ attached to the form $f$ listed in Table~\ref{table:frob-data1997}. 
Let $L$ be the normal closure of the field of $2$-torsion $F(A_f[2])$. Then there is a polynomial $h$ in Table~\ref{table:2-tors-pols1997}
such that $L$ is the splitting field of $h$. In that case, we have 
$$\Gal(L/\Q) \simeq S_3^2 \rtimes \Z/2\Z,\,\,\text{or}\,\, \Gal(L/\Q) \simeq S_4^2 \rtimes \Z/2\Z.$$
\end{thm}

\begin{proof}  Let $G = \Gal(F(A_f[2])/F)$. By Lemma~\ref{lem:sqrt2-tors-fld-sqrt1997} and Theorem~\ref{thm:2-torsion-grps}, there is
a subgroup $H \le (\Z/2\Z)^8$, the unique normal subgroup of $\SL_2(\F_2[\varepsilon])$ of order $8$, such that $G \simeq H \rtimes S_3$. 
Since $F(A_f[\sqrt{2}])$ is Galois over $F$, we see that $F(A_f[2])$ is the compositum of $r$ quadratic extensions $E/F(A_f[\sqrt{2}])$,
where $r$ be the rank of $H$ as an $\F_2$-space. Each quadratic extension $E/F(A_f[\sqrt{2}])$ is unramified outside
the primes above $2$.

Let us write 
$$\Gal(F(A_f[\sqrt{2}])/F) = D_3 =\langle \tau, \sigma | \tau^2 = \sigma^3 = 1, \sigma^2\tau = \tau\sigma \rangle.$$
There are three primes $\gP_1$, $\gP_2, \gP_3$ above $2$ in $F(A_f[\sqrt{2}])$, which are permuted transitively by $\Gal(F(A_f[\sqrt{2}])/F)$,
with  $\mathrm{N}_{F(A_f[\sqrt{2}])/\Q}(\gP_1) = 4$.
Up to relabelling those primes, we can write $(2) = (\gP_1 \gP_2 \gP_3)^2$, where $\gP_2 = \sigma(\gP_1)$, $\gP_3 = \sigma^2(\gP_1)$,
and $\tau(\gP_1) = \gP_1$, $\tau(\gP_2) = \gP_3$.

Let $E/F(A_f[\sqrt{2}])$ be a quadratic extension unramified outside $S := \{\gP_1, \gP_2, \gP_3\}$. Then, the conductor of the compositum of the 
fields in the $\Gal(F(A_f[\sqrt{2}])/F)$-orbit of $E$ is of the form $(\gP_1 \gP_2 \gP_3)^{s_0}$ for some $s_0 \ge 0$. Therefore, the root discriminant of 
$F(A_f[2])$ is equal to
\begin{align*}
\delta_{F(A_f[2])} &= \delta_{F(A_f[\sqrt{2}])} \mathrm{N}_{F(A_f[\sqrt{2}])/\Q}(\gP_1 \gP_2 \gP_3)^{\frac{s}{12\cdot 2^r}} 
= \delta_{F(A_f[\sqrt{2}])} 2^{\frac{s}{2^{r+1}}},
\end{align*}
for some $s \ge 0$.  We recall that $\delta_{F(A_f[\sqrt{2}])} = 2\cdot 1997^{1/2} = 89.3756...$. So, we have
$\delta_{F(A_f[2])} < 4\cdot 1997^{1/2} = 178.7512...$ if and only if one of the following holds:
\begin{itemize}
\item $r = 0$, $s \le 1$;
\item $r = 1$, $s \le 3$;
\item $r = 2$, $s \le 7$;
\item $r = 3$, $s \le 15$. 
\end{itemize}
However, we want $\delta_L < 4\cdot 1997^{1/2} = 178.7512...$. We will show that this in fact implies that $s \le 2$ for $r = 0, 1, 2, 3$. 

Let $K(S,2) = K_S^\times/(K_S^\times)^2$ be the 2-Selmer group, where $K_S^\times$ is the subset of $K = F(A_f[\sqrt{2}])$ of
that are $S$-integral. Then $K(S, 2) \simeq (\Z/2\Z)^{12}$ is an $\F_2$-module equipped with a $\Gal(F(A_f[\sqrt{2}])/F)$-action. 
We view $K(S, 2)$ as an $\F_2[D_3]$-module under this action. Every quadratic extension $E/F(A_f[\sqrt{2}])$ unramified outside 
$S$ corresponds to an element in $K(S,2)$. It is not hard to see that there is a bijection between the following two sets:
\begin{enumerate}[(a)]
\item All compositums $K'$ of quadratic extensions $E/F(A_f[\sqrt{2}])$ unramified outside $S$, which are Galois over $F$;
\item $\F_2[D_3]$-submodules $V$ of $K(S, 2)$. 
\end{enumerate}
For $K'$ a compositum as in (a), $[K':F] = 2^r$, where $V$ is the $\F_2[D_3]$-submodule corresponding to $K'$ and $r = \dim_{\F_2}(V)$. 
The $2$-torsion field $F(A_f[2])$ and its normal closure $L$ arise from an $\F_2[D_3]$-submodule with  $r = \dim_{\F_2}(V) \le 3$,
and we must have $\delta_L < 178.7512...$.

For $r = 0$, Lemma~\ref{lem:sqrt2-tors-fld-sqrt1997} shows that there is a unique field $L$. In that case, $F(A_f[2]) = F(A_f[2])$. For $r = 1, 2, 3$, we determine
all possible fields $L$ as follows:

\begin{enumerate}
\item For each non-trivial element $v \in K(S, 2)$, we find the corresponding quadratic extension $E/F(A_f[\sqrt{2}])$. Then we compute the 
$\Gal(F(A_f[\sqrt{2}])/F)$-orbit of $E$, and the submodule $V$ generated by the $\Gal(F(A_f[\sqrt{2}])/F)$-orbit of $v$;
\item If $\dim_{\F_2}(V) = 1, 2, 3$, then we compute the root discriminant $\delta_E$ of $E$. If $\delta_E < 178.7512...$, then we compute the root discriminant
of the normal closure $L$ of $E$. We note that this can be done from the defining polynomial of $E$ as a local computation without having the compute the
field $L$ itself, which in this case is extremely big.
\item If $\delta_L < 178.7512...$, then we compute all degree $12$ subfields whose normal closure is $L$. 
\end{enumerate}
In total, we obtain that there are:
\begin{enumerate}
\item 3 fields $L$ with $\Gal(L/\Q) \simeq S_3^2 \rtimes \Z/2\Z$; and
\item 7 fields $L$ with $\Gal(L/\Q) \simeq S_4^2 \rtimes \Z/2\Z$.
\end{enumerate}
In Table~\ref{table:2-tors-pols1997}, we give some polynomials of degree $12$ whose splitting fields are the fields $L$.
We also give the Galois group of the relative extension $F(A_f[2])/F$. Each of the field $L$ such that 
$\Gal(L/\Q) \simeq S_4^2 \rtimes \Z/2\Z$ is given by a pair of polynomials $(h_1, h_2)$ such that, letting $g_1$ (resp. $g_2$)
be an irreducible factor of $h_1$ (resp. $h_2$) over $F$, we have $\Gal(K_1/F) \simeq S_4$ and $\Gal(K_2/F) \simeq S_4 \times \Z/2\Z$.
So in those cases, there are two possibilities for the extension $F(A_f[2])/F$. 
\end{proof}

\begin{table}\small\centering
\caption{The possible polynomials giving the $2$-torsion field of the abelian surface $A_f$ defined over $F = \Q(\sqrt{1997})$
with everywhere good reduction attached to the form in Table~\ref{table:frob-data1997}. (Here $G = \Gal(F(A_f[2])/F)$.)}
\label{table:2-tors-pols1997}
\begin{tabular}{@{}lcc@{}}\toprule
\multicolumn{1}{c}{$h$} & $G$ & $\Gal(L/\Q)$\\\midrule
$\quad x^6 - 25x^4 - 50x^3 - 343x^2 - 1372x - 1372$ & $S_3$ & $S_3^2 \rtimes \Z/2\Z$ \\[2ex]
$\quad x^{12} + 7x^{10} - 383x^8 - 1662x^6 - 393x^4 - 3505x^2 - 289$ & \multirow{3}{*}{$D_6$} & \multirow{3}{*}{$S_3^2 \rtimes \Z/2\Z$}\\
$\quad x^{12} - 90x^9 + 1923x^8 - 3932x^7 + 4050x^6 - 664x^5$ & &\\
$\qquad\quad{} + 2573x^4 - 3554x^3 + 1922x^2 + 868x + 196$ & &\\[1.ex]
$\left\{\begin{array}{l} 
\!\!\!x^{12} + 17x^{10} - 151x^8 - 1656x^6 + 10988x^4 - 1104x^2 + 16\\
\!\!\!x^{12} - 17x^{10} - 151x^8 + 1656x^6 + 10988x^4 + 1104x^2 + 16 
\end{array}\right.$ & &\\[2ex]
$\left\{\begin{array}{l}
\!\!\!x^{12} - 49x^{10} - 118x^8 + 371x^6 - 442x^4 + 219x^2 + 1\\
\!\!\!x^{12} + 49x^{10} - 118x^8 - 371x^6 - 442x^4 - 219x^2 + 1
\end{array}\right.$ & &\\[2ex]
$\left\{\begin{array}{l}
\!\!\!x^{12} - 63x^{10} + 597x^8 - 3284x^6 + 2956x^4 - 416x^2 + 16 \\
\!\!\!x^{12} + 63x^{10} + 597x^8 + 3284x^6 + 2956x^4 + 416x^2 + 16
\end{array}\right.$ & &\\[2ex]
$\left\{\begin{array}{l}
\!\!\!x^{12} - 58x^{10} + 1101x^8 - 7548x^6 + 9144x^4 - 1040x^2 + 16 \\
\!\!\!x^{12} + 58x^{10} + 1101x^8 + 7548x^6 + 9144x^4 + 1040x^2 + 16
\end{array}\right.$ & $\!\!\!\begin{array}{c} S_4 \\ S_4\times \Z/2\Z \end{array}\!\!\!$ & $S_4^2 \rtimes \Z/2\Z$\\[2ex]
$\left\{\begin{array}{l}
\!\!\!x^{12} - 49x^{10} + 602x^8 - 1293x^6 + 2390x^4 - 501x^2 + 1\\
\!\!\!x^{12} + 49x^{10} + 602x^8 + 1293x^6 + 2390x^4 + 501x^2 + 1
\end{array}\right.$ & &\\[2ex]
$\left\{\begin{array}{l}
\!\!\!x^{12} - 123x^{10} - 904x^8 - 3079x^6 - 7948x^4 - 7791x^2 - 289\\
\!\!\!x^{12} + 123x^{10} - 904x^8 + 3079x^6 - 7948x^4 + 7791x^2 - 289
\end{array}\right.$ & &\\[2ex]
$\left\{\begin{array}{l}
\!\!\!x^{12} + 2x^{10} + 130x^8 + 352x^6 - 110x^4 - 594x^2 - 49\\
\!\!\!x^{12} - 2x^{10} + 130x^8 - 352x^6 - 110x^4 + 594x^2 - 49
\end{array}\right.$ & &\\
\bottomrule
\end{tabular}
\end{table}

\subsection{The surfaces}

Let $h(x)$ be one of the polynomial in Table~\ref{table:2-tors-pols1997}, and $g(x)$ an irreducible factor over $F$.
Let $Y_{-}(8)$ be the Hilbert modular surface of discriminant $8$. We recall that $Y_{-}(8)$ parametrises principally
polarised abelian surfaces with RM by $\Z[\sqrt{2}]$. Let $\bar{\rho}: \Gal(\Qbar/F) \to \SL_2(\F_2[e])$, with $e^2 = 0$,
be the mod $2$ representation associated to $g$. We let $Y_{-}(8)^{\bar{\rho}}$ be the twist of $Y_{-}(8)$ by $\bar{\rho}$,
this parametrises all pairs $(A, \phi)$ where $A$ is a principally polarised abelian surface defined over $F$ with RM by 
$\Z[\sqrt{2}]$, and $\phi:\,A[2] \simeq \bar{\rho}$ an isomorphism.

\begin{thm}\label{thm:sqrt1997} Assume that there is an abelian surface $A_f$ over $F = \Q(\sqrt{1997})$ attached
to the Hilbert newform $f$ of level $(1)$ and weight $2$ listed in Table~\ref{table:frob-data1997}. Then, there exists a polynomial
$h$ in Table~\ref{table:2-tors-pols1997}, and an irreducible factor $g \in F[x]$ such that $A$ corresponds to an $F$-rational
on the surface $Y_{-}(8)^{\bar{\rho}}$, where $\bar{\rho}: \Gal(\Qbar/F) \to \SL_2(\F_2[e])$, with $e^2 = 0$,
is the mod $2$ representation associated to $g$. 
\end{thm}

\begin{proof} This follows directly from Theorem~\ref{thm:2-tors-fld-sqrt1997}.
\end{proof}

\begin{rem}\rm As we explained earlier in Remark~\ref{rem:missing-exes}, we were unable to find the surface $A_f$ attached to the newform $f$ of level $(1)$ 
over $F = \Q(\sqrt{1997})$ in Table~\ref{table:frob-data1997}, using both the search methods described in Subsections~\ref{subsec:search353}, \ref{subsec:search421}
and~\ref{subsec:search1597}. However, as Theorem~\ref{thm:sqrt1997} indicates, $A_f$ must correspond to an $F$-rational 
point on a twist of $Y_{-}(8)^{\bar{\rho}}$, where $\bar{\rho}: \Gal(\Qbar/F) \to \SL_2(\F_2[e])$ is a mod $2$ representation arising from one of the polynomials in 
Table~\ref{table:2-tors-pols1997}. So, it is possible that a search on the twists $Y_{-}(8)^{\bar{\rho}}$ might be successful. The question of finding explicit 
equations for twists of Hilbert modular surfaces is one that is interesting in its own right. Indeed although such twists have been used in modularity lifting 
methods (see~\cite{ell05, sbt97} for example), they have never been approached algorithmically. So, we hope to return to studying them
in the future, and use them to find the surface $A_f$.
\end{rem}

\begin{rem}\rm Assume that there is an abelian surface $A_f$ over $F = \Q(\sqrt{1997})$ attached
to the Hilbert newform $f$ of level $(1)$ and weight $2$ listed in Table~\ref{table:frob-data1997}. 
Then, $A_f$ is isomorphic to the Jacobian $A'$ of some Richelot curve $C':\, y^2 = \tilde{h}(x)$. 
By~\cite[Lemma 4.1]{bd11}, there is a cubic extension $K'/F$ and a quadratic polynomial $Q(x) \in K'[x]$
such that $C'$ is of the form $$C':\,y^2 = \tilde{h}(x) = \mathrm{Norm}_{K'[x]/F[x]}(Q(x)).$$
By the uniqueness of the field $K$ in Lemma~\ref{lem:sqrt2-tors-fld-sqrt1997}, we see $K'/F$ is one
the cubic extension of $F$ contained in $K$. So, up to Galois conjugation, it is defined by a
cubic factor of the polynomial $h'(x) =x^6 - 25x^4 - 50x^3 - 343x^2 - 1372x - 1372$. Bending~\cite[Theorem 4.1]{ben99}
gives a parametrisation of abelian surfaces with RM by $\Z[\sqrt{2}]$ using the fact that they are Jacobians of Richelot
curves. However, his family does not seem to be very suitable for height search. 
\end{rem}

\begin{rem}\rm
There are six {\it pairwise} non-isogenous elliptic curves over $F$ with trivial conductor.  
Their $\Gal(F/\Q)$-conjugacy classes are represented by the curves:
\begin{align*}
&E_1: y^2 + w x y = x^3 + (w + 1)x^2 + (111w +  5401)x + (2406w + 81112);\\
&E_2: y^2 + (w + 1)xy + (w + 1)y = x^3 - wx^2 + (19636w + 434383)x\\
&\qquad\qquad\qquad\qquad\qquad\qquad\qquad\qquad\qquad\, + (5730650w + 125261893);\\
&E_3: y^2 + wxy + (w + 1)y = x^3 - x^2 + (9370w - 208733)x\\
&\qquad\qquad\qquad\qquad\qquad\qquad\qquad\qquad\qquad\,\,\,\, + (2697263w - 61535794). 
\end{align*}
The fields $F(E_1[2])$ and $F(E_2[2])$ have the same Galois closure, which is the field $K$ in Lemma~\ref{lem:sqrt2-tors-fld-sqrt1997}. 
Assuming that the surface $A_f$ in Theorem~\ref{thm:sqrt1997} exists, then $A_f[2]$ is an extension of $E_1[2]$ or $E_2[2]$ in
the category of finite flat $2$-group schemes defined over $\CO_F$ up to Galois conjugation. Theorem~\ref{thm:2-tors-fld-sqrt1997}
shows that there are least 10 possibilities for the generic fibre of $A_f[2]$. This makes it harder to pin down $A_f[2]$, and hence
the $2$-adic Tate module of $A_f$. In part, this explains the extra difficulties we experienced in finding $A_f$ using the same height
search as in Subsection~\ref{subsec:search353}.
\end{rem}

\providecommand{\bysame}{\leavevmode\hbox to3em{\hrulefill}\thinspace}
\providecommand{\MR}{\relax\ifhmode\unskip\space\fi MR }
\providecommand{\MRhref}[2]{%
  \href{http://www.ams.org/mathscinet-getitem?mr=#1}{#2}
}
\providecommand{\href}[2]{#2}

\end{document}